\documentclass[12pt,letterpaper]{amsart}
\makeatletter
\@namedef{subjclassname@2020}{%
	\textup{2020} Mathematics Subject Classification}
\makeatother
\usepackage{geometry,xcolor}
\usepackage{setspace}
\usepackage{mathpazo} %Fonts
\usepackage[colorlinks,allcolors=blue]{hyperref} %pagebackref=true
\usepackage{xpatch,multirow}
\xpatchcmd{\proof}{\itshape}{\bfseries}{}{}
\newtheorem{theorem}{Theorem}
\newtheorem{proposition}{Proposition}
\newtheorem{corollary}{Corollary}
\newtheorem{lemma}{Lemma}
\newtheorem{conjecture}{Conjecture}
\theoremstyle{remark}

\newcommand{\C}{\mathbb{C}}

\newcommand{\zb}{\overline{z}}
\newcommand{\wb}{\overline{w}}

\usepackage{soul}
\allowdisplaybreaks

\title{Fixed points of the Berezin transform on Fock-type spaces}
\author{Ghazaleh Asghari}
\address[Ghazaleh Asghari]{Department of Mathematics and Statistics, 
University of Reading, England}
\email{g.asgharikhonakdari@pgr.reading.ac.uk}

\author{\v{Z}eljko \v{C}u\v{c}kovi\'c}
\address[\v{Z}eljko \v{C}u\v{c}kovi\'c]{University of Toledo, Department of 
Mathematics \& Statistics, Toledo, OH 43606, USA}
\email{Zeljko.Cuckovic@utoledo.edu}

\author{S\"{o}nmez \c{S}ahuto\u{g}lu}
\address[S\"{o}nmez \c{S}ahuto\u{g}lu]{University of Toledo, Department of 	
Mathematics \& Statistics, Toledo, OH 43606, USA}
\email{Sonmez.Sahutoglu@utoledo.edu}

\thanks{G. Asghari was supported by EPSRC grant EP/W524128/1.}
\date{\today}
\begin{document}
	
\begin{abstract}
We study the fixed points of the Berezin transform on the Fock-type 
spaces $F_m^2$ with the weight $e^{-|z|^m}, m > 0.$  It is known that the 
Berezin transform is well-defined on the polynomials in $z$ and $\overline{z}$.  
In this paper we focus on the polynomial fixed points and we show that these 
polynomials must be harmonic, except possibly for countably many 
$m \in (0, \infty).$  We also show that, in some particular cases, the fixed 
point polynomials are harmonic for all $m.$
\end{abstract} 
\maketitle

The study of operators on the Bergman spaces is closely related 
to the properties of the Berezin transform $B$. This connection 
comes in a natural way through the use of the Bergman kernel, 
which is an essential object in the study of Bergman spaces on 
different domains. The Berezin transform is also an interesting 
object in its own right. A lot of work has been done in studying 
the regularity properties of the Berezin transform 
(see \cite{ArazyEnglis01,Coburn05,Englis07,CuckovicSahutoglu21}), 
as well as studying its range 
(see \cite{Ahern2004,CuckovicLi2012,Rao2018,CowenFelder2022}).

In this paper we are interested in finding the fixed points of the Berezin 
transform. Characterizing functions that satisfy the 
equality $Bf=f$ is an interesting and deep problem. The known 
results show the connection of these functions and different 
notions of harmonicity. 
One of the first work in this direction was the paper by 
Axler and \v{C}u\v{c}kovi\'c \cite{AxlerCuckovic1991}
who studied this problem in connection with the problem of 
commuting Toeplitz operators on the Bergman space on the unit disk. 
Engli\v{s} \cite{Englis1994} showed that on the unit disk, bounded fixed points 
are precisely harmonic functions. Several authors have continued this 
line of investigation, we mention \cite{AhernFloresRudin1993}, Lee 
\cite{Lee2008,Lee2017}, Arazy and Engli{\v{s}} \cite{ArazyEnglis01}, 
Jevti\'c \cite{Jevtic2003}, and Casseli \cite{Casseli2020}, among others.

On the Fock space $F^{2}_{2}$, the situation is different.  
Engli\v{s} \cite{EnglisThesis,Englis1994}  (see also \cite[section 3.3]{ZhuBookFock}) 
showed that there are non harmonic fixed points of the corresponding 
Berezin transform. For example $f(x+iy)=e^{ax+by}$ is fixed in $F^{2}_{2}$ 
for any $a,b\in\C$ such that $a^{2}+b^{2}=8\pi i$, however $f$ is not 
harmonic. On the other hand, if $f$ is a bounded fixed point of the 
Berezin transform on $F^{2}_{2}$, then it has to be a constant. 
In this paper we are interested in studying the fixed point problem 
on the family of Fock-type spaces $F^{2}_{m}$.

Let $m$ be a positive real number. Consider the space 
$L^{2}_{m}=L^{2}(\C,e^{-|z|^{m}}dA)$, where 
$dA=rdrd\theta/2\pi$ is a multiple of the Lebesgue 
measure on $\C$. $L^{2}_{m}$ is a Hilbert space with the inner product
\[\langle f,g\rangle=\int_{\C}f(z)\overline{g(z)}e^{-|z|^{m}}dA(z).\]
The Fock-type space $F^{2}_{m}$ is the closed subspace of entire 
functions inside $L^{2}_{m}$. It is a reproducing kernel Hilbert space 
with kernel, called the Bergman kernel, given by 
\begin{align*}
K_m(w,z)=m\sum_{k=0}^{\infty} \frac{w^{k}\zb^{k}}{\Gamma(\frac{2k+2}{m})}.
\end{align*}
Let $k_{m,z}=\frac{K_{m,z}}{\|K_{m,z}\|}$ be the normalized 
Bergman kernel, where $K_{m,z}(w)=K_m(w,z)$. One can define the Berezin 
transform of a function $f$ as
\begin{align*} 
B_{m}f(z)=\langle fk_{m,z},k_{m,z}\rangle
=\int_{\C} f(w) |k_{m,z}(w)|^{2}e^{-|w|^m}dA(w),
\end{align*}
whenever the above integral exists. We note that the estimate in the 
proof of  \cite[Lemma 5.2]{BHYoussfi2007} implies that the Berezin 
transform is defined for any polynomial in $z$ and $\zb$. For more 
information about Fock-type spaces, we refer the reader 
to \cite{BH2010,BHYZhu2017}. 

In view of the counterexample above, we will study the polynomials
in $z$ and $\zb$ that are fixed points of the Berezin transform on  
Fock-type spaces. First, we show that if a polynomial with non-negative 
coefficients is fixed, then it has to be harmonic.

%%%%%%%%%%%%%%%%%%%%%%%%%%%%%%%%%%%
\begin{theorem}\label{ThmPositive} 
Assume that $B_mf=f$ for a polynomial $f$ of $z$ and $\zb$  
with nonnegative coefficients and for some $m>0$. Then $f$ is harmonic.  
\end{theorem}

Next, we show that on $F^{2}_{2}$, if $B_{2}f=f$ and $f$ is a 
polynomial, then $f$ must be harmonic. 

%%%%%%%%%%%%%%%%%%%%%%%%%%%%
\begin{theorem}\label{ThmCase2}
Let $f$ be a polynomial of $z$ and $\zb$ such that $B_2f=f$. 
Then $f$ is harmonic.	
\end{theorem} 

For $m\neq 2$, the situation is more difficult, since the Bergman  
kernels do not have a closed form and the computations involve 
many gamma functions. Our main result shows that $B_{m}f=f$ 
for a polynomial $f$ implies that $f$ is harmonic for all $m$, 
except possibly a countably many $m$. 
We denote the non-negative integers as $\mathbb{N}_0$.

%%%%%%%%%%%%%%%%%%%%%%%
\begin{theorem}\label{ThmPoly}
For $n\in\mathbb{N}_0$, there exists a discrete (possibly empty) 
set $Z_n\subset (0,\infty)$ with no cluster points in $(0,\infty)$ such that if 
$m\in (0,\infty)\setminus Z_n$ and $B_mf=f$ for a polynomial $f$ of 
degree at most $n$, then $f$ is harmonic. 
\end{theorem} 

Since harmonic polynomials are fixed points of Berezin transform 
(see Lemma \ref{LemHarmonic}) we have the following corollary. 
	
\begin{corollary}
Let $f$ be a polynomial of $z$ and $\zb$ of degree at most $n$.  
Then the following are equivalent. 
\begin{itemize}
\item[i.)] $B_mf=f$ for some $m\in (0,\infty)\setminus Z_n$, 
\item[ii.)]  $B_mf=f$ for all $m\in (0,\infty)\setminus Z_n$. 
\end{itemize} 
\end{corollary} 
	
Since countable union of countable sets is countable, we have the 
following corollary.

\begin{corollary}\label{Cor2}
There exists a countable  (possibly empty) set $Z\subset (0,\infty)$ 
such that if $B_mf=f$ for a polynomial $f$ of $z$ and $\zb$ and 
$m\in (0,\infty)\setminus Z$, then $f$ is harmonic. 
\end{corollary} 	

By using a more computational approach, we also show that if 
$B_{m}f=f$  for a binomial function $f$, then $f$ is harmonic. 

\begin{theorem}\label{ThmBino}
Let $m>0$ and $f(z)=c_{1}z^{a}\zb^{b}+c_{2}z^{c}\zb^{d}$, 
where $a$, $b$, $c$, and $d$ are  positive integers. 
Then $f$ is a fixed point of the Berezin transformation $B_{m}$ 
if and only if $c_{1}=c_{2}=0$.
\end{theorem}

It would be interesting to know if the set $Z$ in Corollary \ref{Cor2} 
can be non-empty. 
\begin{conjecture}
Let  $f$ be a polynomial of $z$ and $\zb$.  If for any $m>0$ we have $B_mf=f$, 
then $f$ is harmonic.
\end{conjecture}

%%%%%%%%%%%
\section*{Preliminaries}

For $m$ a positive real number, the space $F^{2}_{m}$ is defined 
as a subspace of holomorphic functions in $L^{2}(\C, e^{-|z|^{m}}dA)$. 
We note that,  throughout the paper, we will use   
\[dA(z)=\frac{rdrd\theta}{2\pi}.\]
Given a function $f\in F^{2}_{m}$, its Taylor series 
$f(z)=\sum_{j=0}^{\infty}f_{j}z^{j}$ converges uniformly on compact 
subsets of $\C$. Furthermore, one can check that
\begin{align*}
\|f\|^{2}=\sum_{j=0}^{\infty}|f_{j}|^{2} \frac{1}{m}
\Gamma\left(\frac{2j+2}{m}\right) 
\end{align*}
and
\begin{align*}
\langle f,g\rangle =\sum_{j=0}^{\infty}f_{j}\overline{g}_j\frac{1}{m}
\Gamma\left(\frac{2j+2}{m}\right).
\end{align*}
Hence, the monomials $z^{n}$, with $n=0,1,2,\cdots$ form an orthogonal 
basis for $F^{2}_{m}$. In particular,
\begin{equation*}
\left\{\sqrt{md_j}z^j;\quad j=0,1,2,\cdots\right\}
\end{equation*}
is an orthonormal basis for $F^2_m$ where 
\begin{align}\label{EqnD_j}
d_j=\frac{1}{\Gamma\left(\frac{2j+2}{m}\right)}.
\end{align} 

To find the Bergman kernel, we proceed as follows. Let $f\in F^{2}_{m}$ 
and $z\in\C$. Then
\begin{align*}
|f(z)|=|\sum_{j=0}^{\infty}f_{j}z^{j}|
&\leq \sum_{j=0}^{\infty}|f_{j}||z|^{j}\\
	&= \sum_{j=0}^{\infty}\frac{|f_{j}|}{\sqrt{md_j}} \sqrt{md_j}|z|^{j}\\
&\leq \left(\sum_{j=0}^{\infty}\frac{|f_{j}|^2}{md_j}\right)^{1/2}
	\left(\sum_{j=0}^{\infty} md_j|z|^{2j} \right)^{1/2}.
\end{align*}
The first sum on the right hand side equals $\|f\|$ and the ratio test 
shows that the second sum above converge uniformly on compact subsets.  
Thus, the evaluation map $f\mapsto f(z)$ is a bounded linear functional 
on $\C$ and uniformly bounded for $z$. Furthermore, 
$F^{2}_{m}$ is closed inside $L^{2}(\C,e^{-|z|^{m}}dA)$, 
and hence a Hilbert space. Thus, there exists $K_{m,z}\in F^{2}_{m}$ 
such that for any $f\in F^{2}_{m}$, $f(z)=\langle f, K_{m,z}\rangle$. Indeed, 
\[f(z)=\sum_{j=0}^{\infty}f_{j}z^{j}
=\sum_{j=0}^{\infty}f_jmd_jz^j\frac{1}{md_j}=\langle f, K_{m,z}\rangle,\]
where, 
$K_{m,z}(w)=m\sum_{j=0}^{\infty}d_jw^j\zb^j$, 
for any $z,w\in\C$. From now on, we will write
\begin{align*} 
K_m(w,z)=K_{m,z}(w) =m\sum_{j=0}^{\infty}d_jw^j\zb^j.
\end{align*}
For more information about Fock-type spaces we refer the reader to 
\cite{BHYoussfi2007,BHYZhu2017}. 

We finish this section with the following property about the beta function
\begin{equation*}
	\beta(x,y)=\int_{0}^{1}t^{x-1}(1-t)^{y-1}dt
\end{equation*}
for $x,y>0$. 
\begin{lemma}\label{LemBeta}
The function $x\to\beta(x,k-x)$ is convex on $(0,k)$ and attains 
its minimum at $x=k/2$.
\end{lemma}

\begin{proof}
By definition,
\[\beta(x,k-x)=\int_{0}^{1}t^{x-1}(1-t)^{k-x-1}dt
=\int_{0}^{1}\left(\frac{t}{1-t}\right)^{x-1}(1-t)^{k-2}dt.\]
Writing $(\frac{t}{1-t})^{x-1}=e^{(x-1)\log (\frac{t}{1-t})}$, we can 
compute the partial derivative as
\[\frac{\partial}{\partial x}\beta(x,k-x)
=\int_{0}^{1}\log\left(\frac{t}{1-t}\right) 
\left(\frac{t}{1-t}\right)^{x-1}(1-t)^{k-2}dt.\]  
It is easy to see that $x=\frac{k}{2}$ is a critical point. Indeed, 
taking $u=1-t$ for $1/2<t<1$, 
\begin{align*}
\frac{\partial}{\partial x}\beta\left(\frac{k}{2},\frac{k}{2}\right) 
=& \int_{0}^{1/2}\log\left(\frac{t}{1-t}\right) t^{k/2-1}(1-t)^{k/2-1}dt \\
&+ \int_{1/2}^{1}\log\left(\frac{t}{1-t}\right) t^{k/2-1}(1-t)^{k/2-1}dt\\
=&\int_{0}^{1/2}\log\left(\frac{t}{1-t}\right) t^{k/2-1}(1-t)^{k/2-1}dt \\
&-\int_{1/2}^{0}\log\left(\frac{1-u}{u}\right) (1-u)^{k/2-1}u^{k/2-1}du\\
=&\int_{0}^{1/2}\log\left(\frac{t}{1-t}\right) t^{k/2-1}(1-t)^{k/2-1}dt\\
&-\int_{0}^{1/2}\log\left(\frac{u}{1-u}\right) (1-u)^{k/2-1}u^{k/2-1}du
=0.
\end{align*}
Computing
\begin{equation*}
\frac{\partial^{2}}{\partial x^{2}}\beta(x,k-x) 
= \int_{0}^{1}\left(\log\left(\frac{t}{1-t}\right)\right)^2 
\left(\frac{t}{1-t}\right)^{x-1}(1-t)^{k-2}dt>0,
\end{equation*}
implies that the beta function is convex on the line $x+y=k$. Hence, 
we can conclude that $x=y=k/2$ is the minimum of the beta function 
on the line $x+y=k$.
\end{proof}

%%%%%%%%%%%%%%%%%%%%%%%%%%
\section*{Preparatory Results}
Let $f$ be a polynomial in $z$ and $\zb$ 
(or a function so that the following integrals make sense). Then 
\begin{equation}\label{EqnBerezin}
\begin{split}
B_{m}f(z)&=\int_{\C}f(w)|k_{m,z}(w)|^{2}e^{-|w|^{m}}dA(w)\\
&=\frac{1}{K_{m}(z,z)}\int_{\C}f(w)|K_{m}(w,z)|^{2}e^{-|w|^{m}}dA(w)\\
&=\frac{m^{2}}{K_{m}(z,z)}\sum_{k,l=0}^{\infty}d_{k}d_{l}z^{k}\zb^{l}
\int_{\C}f(w)\wb^{k}w^{l}e^{-|w|^{m}}dA(w)\\
&=\frac{m^{2}}{K_{m}(z,z)}\sum_{k,l=0}^{\infty}
d_{k}d_{l}z^{k}\zb^{l}\lambda^{f}_{k,l},
\end{split}
\end{equation}
where  $d_j$ is defined is \eqref{EqnD_j} and 
\begin{equation*}
\lambda^{f}_{k,l}=\int_{\C}f(w)\wb^{k}w^{l}e^{-|w|^{m}}dA(w).
\end{equation*}

The following lemma is very easy to show so we will skip the proof. 

\begin{lemma}\label{LemHarmonic}
Harmonic polynomials are fixed points of the Berezin transform.  
\end{lemma}

Let us define 
\begin{equation*}
H_{n,\tau}=\left\{\sum_{j=0}^n a_jz^{j+\tau}\zb^j:a_j\in \C\right\}
\end{equation*}
for $n,\tau\in\mathbb{N}_{0}$, and 	
\begin{equation*}
H_{n,\tau}=\left\{\sum_{j=0}^n a_jz^j\zb^{j-\tau}:a_j\in \C\right\}
\end{equation*}
for $\tau\in\mathbb{Z}\setminus \mathbb{N}_0$ and $n\in \mathbb{N}_0$. 
Let $C^{\omega}(\C,e^{-|z|^m}dA)$ denote the real analytic functions 
in $L^2(\C,e^{-|z|^m}dA)$ and $A_{\tau}\subset C^{\omega}(\C,e^{-|z|^m}dA)$ 
denote the set of real analytic functions whose $n$th Taylor polynomial 
belong to $H_{n,\tau}$ for all $n$. We note that $A_{\tau}\cap A_s=\{0\}$ 
for $s\neq \tau$. 

\begin{lemma} \label{LemAtau}
Let $n\in\mathbb{N}_0,\tau\in\mathbb{Z},$ and $m>0$. 
Then $B_mf \in A_{\tau}$ and $K_m(\cdot,\cdot)B_mf \in A_{\tau}$ 
for all $f\in H_{n,\tau}$.
\end{lemma} 

\begin{proof}
Since $K_m(z,z)=m\sum_{k=0}^{\infty}d_k|z|^{2k}\in A_0$, it is enough 
to show that  $K_m(\cdot,\cdot)B_m$ maps a monomial in $H_{n,\tau}$ 
into $A_{\tau}$. Since $H_{n,\tau}$ is composed of polynomials, it is 
enough to prove the lemma for monomials. So, without loss of generality, 
let $\tau\geq 0$ and $f(z)=z^{\alpha+\tau}\zb^{\alpha}$ for some 
$\alpha\in\mathbb{N}_0$. Then by \eqref{EqnBerezin} we have 
\begin{equation*}
K_m(z,z)B_mf(z)=m^2\sum_{k,l=0}^{\infty}d_kd_lz^k\zb^l\lambda^f_{k,l}, 
\end{equation*}
where 
\begin{equation*}
\lambda^f_{k,l} 
=\int_{\C} w^{\alpha+\tau}\wb^{\alpha}\wb^{k}w^{l} e^{-|w|^m}dA(w)
=\int_{\C} w^{\alpha+\tau+l}\wb^{\alpha+k} e^{-|w|^m}dA(w).
\end{equation*}
Taking $w=re^{i\theta}$ and the normalized measure 
$dA(w)=\frac{1}{2\pi}rdrd\theta$, it is easy to see that 
the above integral is nonzero only if $k=l+\tau$. Then
\begin{equation}\label{EqnBerezinMonomial}
\begin{split}
K_{m}(z,z)B_{m}f(z)
&=m^{2}\sum_{l=0}^{\infty}d_{l+\tau}d_{l}z^{l+\tau}\zb^{l}
	\int_{0}^{\infty}r^{2(\alpha+\tau+l)+1}e^{-r^{m}}dr\\
&=m\sum_{l=0}^{\infty}d_{l+\tau}d_{l}z^{l+\tau}\zb^{l}
	 \Gamma\left(\frac{2(\alpha+\tau+l)+2}{m}\right)\\
&=m\sum_{l=0}^{\infty}\frac{d_{l+\tau}d_{l}}{d_{\alpha+\tau+l}}
	z^{l+\tau}\zb^{l}\in A_{\tau}.
\end{split}
\end{equation}
Therefore, one can use long division to show that 
$B_{m}f\in A_{\tau}$ as $K_m\in A_0$. 
\end{proof}

\begin{lemma}\label{LemSum} 
Let $f$ be a polynomial (of $z$ and $\zb$) of degree $n$ such that 
$B_m f=f$. Then $B_mf_{\tau}=f_{\tau}$ for $-n\leq \tau\leq n$ 
where $f=\sum_{\tau=-n}^nf_{\tau}$ and $f_{\tau}\in H_{n,\tau}$.
\end{lemma} 	

\begin{proof} 
Since $f$ is a polynomial (of $z$ and $\zb$) of degree $n$, 
we have a unique decomposition 	
\begin{equation*}
f=\sum_{\tau=-n}^nf_{\tau},
\end{equation*}
where $f_{\tau}\in H_{n,\tau}$. We assume that $B_mf=f$. Then, 
by Lemma \ref{LemAtau}, we have $B_mf_{\tau}\in A_{\tau}$ 
for $-n\leq \tau\leq n$. Hence,  
\begin{equation*}
\sum_{\tau=-n}^nf_{\tau} =f=B_mf =\sum_{\tau=-n}^nB_mf_{\tau},
\end{equation*}
implying that $B_mf_{\tau}=f_{\tau}$. 
\end{proof} 

%%%%%%%%%%%%%%%%%%%%%%%%%%%%%
\section*{Proof of Theorem \ref{ThmPositive}}
\begin{proof}[Proof of Theorem \ref{ThmPositive}] 
By Lemma \ref{LemSum}, it is enough to prove the theorem 
for $f\in H_{n,\tau}$, where $-n\leq \tau\leq n$. 
Since, by Lemma \ref{LemHarmonic} harmonic polynomials are fixed, 
without loss of generality, we assume that $B_mf=f$ where 
$f(z)=\sum_{j=1}^{n}a_{j}z^{j+\tau}\zb^{j}\in H_{n,\tau}$ 
is a  polynomial with $a_j\geq 0$ for all $j\geq 1$,  
$a_n>0$, and $\tau\geq 0$. Hence, $f$ is not harmonic. 
Similar to \eqref{EqnBerezinMonomial}, one obtains that 
\begin{equation*}
K_{m}(z,z)B_mf(z)=m\sum_{l=0}^{\infty}\sum_{j=1}^{n} a_{j}
\frac{d_{l+\tau}d_{l}}{d_{j+l+\tau}}z^{l+\tau}\zb^{l},
\end{equation*}
and 
\[K_{m}(z,z)f(z)
=m\sum_{l=0}^{\infty}\sum_{j=1}^{n}a_{j}d_{l}z^{j+\tau+l}\zb^{j+l}
=m\sum_{l=j}^{\infty}\sum_{j=1}^{n}a_{j}d_{l-j}z^{l+\tau}\zb^{l}.\]
Note that by definition and Lemma \ref{LemAtau}, $K_{m}(\cdot,\cdot)f$ 
and $K_{m}(\cdot,\cdot)B_{m}f$ both belong to $A_\tau$, and hence their 
difference as well.  
\begin{align}\nonumber 
0=K_{m}(z,z)B_mf(z)-K_{m}(z,z)f(z)
= & m\sum_{l=0}^{j-1}\sum_{j=1}^{n}a_{j}
\frac{d_{l+\tau}d_{l}}{d_{j+\tau+l}}z^{l+\tau}\zb^{l}\\
\label{EqnZero}&+m\sum_{l=j}^{\infty}\sum_{j=1}^{n}a_{j}
\left(\frac{d_{l+\tau}d_{l}}{d_{j+\tau+l}}-d_{l-j}\right)z^{l+\tau}\zb^{l}
\end{align}
for all $z$. We note that $d_{j+\tau+l}$ is nonzero for any $j\geq 1$ 
because the gamma function has poles only on negative integers. 
We will use the following well known formula
\[\beta(a,b)=\frac{\Gamma(a)\Gamma(b)}{\Gamma(a+b)}.\]
Taking $x=\frac{2}{m}$, and 
\[d_{l}=\frac{1}{\Gamma\left(\frac{2(l+1)}{m}\right)}=\frac{1}{\Gamma((l+1)x)},\] 
for every nonnegative $j$ we obtain
\begin{align*}
\frac{d_{l+\tau}d_{l}}{d_{j+\tau+l}}-d_{l-j}
= \frac{d_{l+\tau}d_{l}-d_{l-j}d_{j+\tau+l}}{d_{j+\tau+l}} 
= \Gamma((j+\tau+l+1)x)\Gamma((2l+\tau+2)x)
\frac{B_{l,j,\tau}}{A_{l,j,\tau}} 
\end{align*}
where 
\begin{align*}
A_{l,j,\tau}=&\Gamma((l+\tau+1)x)
\Gamma((l+1)x)\Gamma((l-j+1)x)\Gamma((j+\tau+l+1)x),\\
B_{l,j,\tau}=&\beta((j+l+\tau+1)x,(l-j+1)x)-\beta((l+\tau+1)x,(l+1)x).
\end{align*}  
In this case  
\[k=(j+l+\tau+1)x+(l-j+1)x=(l+\tau+1)x+(l+1)x=(2l+\tau+2)x.\] 
By Lemma \ref{LemBeta},  $\beta(y,k-y)$ is a convex function of $y$ and takes its 
minimum at $y=k/2$. 
Hence $\beta(\alpha_1,k-\alpha_1)> \beta(\alpha_2,k-\alpha_2)$ 
if $\alpha_1> \alpha_2\geq k/2$. 
We choose $\alpha_1=(j+l+\tau+1)x$ and $\alpha_2=(l+\tau+1)x$. Then 
\[(j+l+\tau+1)x>(l+\tau+1)x>(l+\frac{\tau}{2}+1)x=\frac{k}{2}.\]
Moreover, $k-\alpha_1=(l-j+1)x$ and $k-\alpha_2=(l+1)x$. 
Then for any $j\geq 1$, we have 
\[\beta((j+\tau+l+1)x,(l-j+1)x)-\beta((l+\tau+1)x,(l+1)x)>0.\]
Hence, $a_n>0$ implies that 
\[K_{m}(z,z)B_mf(z)-K_{m}(z,z)f(z)\not\equiv 0\] 
which contradicts \eqref{EqnZero}. Therefore, $f$ is not a fixed 
point of the Berezin transform.
\end{proof} 

%%%%%%%%%%%%%%%%%%%%%%%%%%
\section*{Proof of Theorem \ref{ThmCase2}}
When $m=2$, one can write the Berezin transform as
\[B_2f(z)=2\int_{\C} f(z+\xi)e^{-|\xi|^2}dA(\xi)\]
where $dA=rdrd\theta/2\pi$. 
We note that the since we normalize $dA$ by $2\pi$, our formula 
above has a multiple 2 instead of $1/\pi$ as in  \cite[section 3.2]{ZhuBookFock}. 
\begin{lemma}\label{LemM=2}
Let $f(z)=z^{j+\tau}\zb^j$ for $j\in \mathbb{N}$ and $\tau\in\mathbb{N}_0$. 
\[B_2f(z)= z^{j+\tau}\zb^j+\frac{(j+\tau)j}{d_1}z^{j-1+\tau}\zb^{j-1}
+\textit{lower order terms}. \]
Furthermore, $B_2:H_{n,\tau}\to H_{n,\tau}$ is a bijection.  
\end{lemma} 
\begin{proof} 
By the binomial expansion formula, we compute the Berezin 
transform for $f(z)=z^{j+\tau}\zb^j$ as follows. 
\begin{align*}
B_2f(z)=&\, 2 \int_{\C} f(z+\xi)e^{-|\xi|^2}dA(\xi)\\
=&\, 2\sum_{s=0}^j\sum_{t=0}^{j+\tau} \binom{j+\tau}{t}\binom{j}{s}z^t\zb^s
\int_{\C} \xi^{j+\tau-t}\overline{\xi}^{j-s}e^{-|\xi|^2}dA(\xi)\\
(t=s+\tau)\quad = &\,\frac{1}{\pi}\sum_{s=0}^j\binom{j+\tau}{s+\tau}
\binom{j}{s}z^{s+\tau}\zb^s  \int_{0}^{2\pi}\int_{0}^{\infty} 
	r^{2(j-s)+1} e^{-r^2}drd\theta\\
=&\, \sum_{s=0}^{j} \binom{j+\tau}{s+\tau}
\binom{j}{s}\Gamma(j-s+1)z^{s+\tau}\zb^s\\
=&\, z^{j+\tau}\zb^j+\frac{(j+\tau)j}{d_1}z^{j-1+\tau}\zb^{j-1}
+\textit{lower order terms},
\end{align*}
where $d_1^{-1}=\Gamma(2)$, as defined before. Furthermore, 
the formula above shows that $B_2f\in H_{n,\tau}$ whenever 
$f\in H_{n,\tau}$. 

Let us choose $B=\{z^{\tau},z^{1+\tau}\zb,\ldots,z^{n+\tau}\zb^n\}$ 
as a basis for $H_{n,\tau}$. Then the matrix representation $[B_2]$ 
for $B_2$ with respect to basis $B$ can be computed as follows.
\begin{align}\label{EqnMatrix}
[B_2]=\left[ \begin{matrix} 1 & (1+\tau)/d_1&\cdots &\cdots &\cdots \\ 
	0& 1&  (2+\tau)2/d_1&\cdots &\cdots \\
	0&0&1&(3+\tau)3/d_1&\cdots \\
	\vdots &\vdots &\vdots &\vdots  &\\
	0&\cdots &\cdots &0&1 \end{matrix} \right]_{(n+1)\times(n+1)}.
\end{align} 
Therefore, $B_{2}$ is a bijection on $H_{n,\tau}$ because the 
matrix $[B_2]$ is nonsingular.  
\end{proof} 

%%%%%%%%%%%%%%%%%%%%%%%%
\begin{proof}[Proof of Theorem  \ref{ThmCase2}]
Let $f$ be a polynomial of $z$ and $\zb$ of order $n$. Then,
by Lemma \ref{LemSum}, it is enough to prove the theorem for
$f\in H_{n,\tau}$, where $-n\leq \tau\leq n$. Without loss of generality, 
we fix $\tau\geq 0$. 

Using the notation in the proof of Lemma \ref{LemM=2}, we observe 
that the matrix $[B_2]$ in \eqref{EqnMatrix} is an upper 
triangular matrix with 1 on the diagonal and $(j+\tau)j/d_1$ on the entries  
above the diagonal. Hence, $[B_2]-I$ is an upper triangular matrix 
with 0 on the diagonal and $(j+\tau)j/d_1$ on the entries above the 
diagonal. Then the first column and the last row of $[B_2]-I$ 
are composed of zero entries.  That is 
\[[B_2] - I = \left[
\begin{array}{c|c}
0 & \multirow{3}{*}{\smash{\raisebox{0.5ex}{$M$}}} \\
\vdots & \\
0 & \\
\hline
0 & \cdots \; 0
\end{array}
\right],\]
where 
\[M=\left[ \begin{matrix}  (1+\tau)/d_1&\cdots &\cdots &\cdots \\ 
0&  (2+\tau)2/d_1&\cdots &\cdots \\
0&0&(3+\tau)3/d_1&\cdots \\
\vdots &\vdots &\vdots &\vdots  \\
0&0&0&(n+\tau)n/d_1 \\
\end{matrix} \right]_{n\times n}\]
is the $n\times n$ submatrix of $[B_2]-I$. We note that $M$ is upper
 triangular with positive entries on the diagonal.
Hence $M$ is of rank $n$ as $\det(M)>0$. That is, $rank([B_2]-I)\geq n$. 
We also know that $z^{\tau}$ is in the kernel of $[B_2]-I$. 
The rank-nullity theorem implies that 
\[rank([B_2]-I)+dim(ker([B_2]-I))=n+1.\]
Then, $rank([B_2]-I)= n$ and $dim(ker([B_2]-I))=1$. Therefore, 
if $f\in H_{n,\tau}$ such that $B_2f=f$, then $f\in ker([B_2]-I)$. 
Namely, $f$ is holomorphic. 
\end{proof} 

%%%%%%%%%%%%%%%%%%%%%%%%%%%
\section*{Proof of Theorem \ref{ThmPoly}}
\begin{proof}[Proof of Theorem \ref{ThmPoly}]
By Lemma \ref{LemSum}, it is enough to prove the theorem for functions 
in $H_{n,\tau}$, where $0\leq \tau\leq n$. Let us define 
\[B_{m,n}:H_{n,\tau}\to C^{\omega}(\C,e^{-|z|^m}dA)\] 
to be the Berezin transform of $F^2_m$ restricted to $H_{n,\tau}$ and 
$T_{m,n}:H_{n,\tau}\to C^{\omega}(\C,e^{-|z|^m}dA)$ as 
\[T_{m,n}f(z)=K_m(z,z)B_{m,n}f(z)-K_m(z,z)f(z)\]
for $f\in H_{n,\tau}$ and $z\in \C$. Then
\[ker(B_{m,n}-I)=ker(T_{m,n})\supseteq span\{z^{\tau}\}=ker(T_{2,n}),\] 
and 
\[rank(B_{m,n}-I)=rank(T_{m,n})\leq n=rank(T_{2,n}).\] 
Then rank-nullity theorem implies that 
\[rank(T_{m,n})+dim(ker(T_{m,n}))=dim(H_{n,\tau})=n+1\] 
for all $m>0$. 

Let $f(z)=z^{j+\tau}\zb^{j}\in H_{n,\tau}$. Then by 
\eqref{EqnBerezinMonomial} we have 
\begin{align*}
K_{m}(z,z)B_mf(z) 
=m\sum_{l=0}^{\infty}\frac{d_{l+\tau}d_{l}}{d_{j+\tau+l}}z^{l+\tau}\zb^{l}.
\end{align*}
Moreover, $K_{m}(z,z)f(z)=m\sum_{l=0}^{\infty}d_{l}z^{l}\zb^{l}z^{j+\tau}\zb^{j}$, 
and hence
\begin{align*}
T_{m,n}(z^{j+\tau}\zb^{j})&=m\sum_{l=0}^{\infty}
\frac{d_{l+\tau}d_{l}}{d_{j+\tau+l}}z^{l+\tau}\zb^{l}
-m\sum_{l=0}^{\infty}d_{l}z^{j+\tau+l}\zb^{j+l}\\
&=m\sum_{l=0}^{\infty}\frac{d_{l+\tau}d_{l}}{d_{j+\tau+l}}z^{l+\tau}\zb^{l}
- m\sum_{l=j}^{\infty}d_{l-j}z^{l+\tau}\zb^{l}\\
&=m\sum_{l=0}^{j-1}\frac{d_{l+\tau}d_{l}}{d_{j+\tau+l}}z^{l+\tau}\zb^{l}
+m\sum_{l=j}^{\infty}\left(\frac{d_{l+\tau}d_{l}}{d_{j+\tau+l}}-d_{l-j}\right)
z^{l+\tau}\zb^{l}.
\end{align*} 
Hence, 
\[T_{m,n}(z^{j+\tau}\zb^j)=\sum_{k=0}^{\infty}a_{j,k}(m)z^{k+\tau} \zb^k,\]
where each $a_{j,k}$ is holomorphic on  $U=\{z\in \C:Re(z)>0\}$, 
by properties of the gamma function. 
Then we will study the rank and nullity of the matrix 
$[a_{j,k}(m)]$ as a function of the complexified variable $m$. 
We note that the matrix $[a_{j,k}(m)]$  is of size $\infty\times (n+1)$.  

Since, by Theorem \ref{ThmCase2}, $rank(T_{2,n})=n$, there exists a 
submatrix $S_{n}$ of $[a_{j,k}]$ of size $n\times n$ with entries 
holomorphic on  $U$ such that $\det(S_{n}(2))\neq 0$. 
Let $S(m)=\det(S_n(m))$ and $Z_{n,\tau}$ denote the zero set of $S$. 
Then $S$ is holomorphic on $U$ and $Z_{n,\tau}$ is a discrete set with no 
accumulation point in $U$. Hence $rank(T_{m,n})\geq n$ for 
$z\in U\setminus Z_{n,\tau}$ as $\det(S_n(m))\neq 0$ for 
$z\in U\setminus Z_{n,\tau}$. However, $rank(T_{m,n})\leq n$ 
as $dim(ker(T_{m,n}))\geq 1$ for all $0<m<\infty$. Then, 
$rank(T_{m,n})=n$ and $dim(ker(T_{m,n}))=1$ for 
$m\in (0,\infty)\setminus Z_{n,\tau}$. Namely, 
$span\{z^{\tau}\}=ker(T_{m,n})$ for $m\in (0,\infty)\setminus Z_{n,\tau}$. 

Let $Z_n=\cup_{\tau=-n}^nZ_{n,\tau}$. Hence $Z_n$ is a discrete set 
with no cluster in $(0,\infty)$. Furthermore, we showed that only the 
holomorphic polynomial is fixed whenever $0\leq \tau \leq n$ and 
$m\in [0,\infty)\setminus Z_n$. Similarly, we can show that  only the 
conjugate holomorphic polynomial is fixed whenever $-n\leq \tau < 0$. 
Therefore, if  $m\in (0,\infty)\setminus Z_n$ and $B_mf=f$ for a 
polynomial $f$ of degree at most $n$, then $f$ is harmonic.
\end{proof}

%%%%%%%%%%%%%%%%%%%%%%%%%%%%%%%%%%%
\section*{Proof of Theorem \ref{ThmBino}}
First we focus on a single monomial.
\begin{proposition}\label{PropMono}
Let $m>0$ and $f(z)=z^{p}\zb^{q}$ be a fixed point for the Berezin 
transform $B_{m}$. Then either $p$ or $q$ must be zero. That is, $f$ 
is either a holomorphic or a conjugate holomorphic monomial.
\end{proposition}
\begin{proof} 
By Lemma \ref{LemHarmonic}, $z^{p}$ and $\zb^{q}$ are fixed under 
$B_{m}$. We assume that $B_{m}f=f$ and 
$p,q\geq 1$. Then \eqref{EqnBerezin} implies that
\begin{equation*}
K_{m}(z,z)z^{p}\zb^{q}=K_{m}(z,z)B_mf(z)
=\int_{\mathbb{C}}w^{p}\wb^{q}|K_{m}(w,z)|^{2}e^{-|w|^{m}}dA(w).
\end{equation*}
Hence,
\begin{align*}
m\sum_{k=0}^{\infty}d_{k}z^{p}\zb^{q}z^{k}\zb^{k}
&=m^{2}\sum_{k,l=0}^{\infty}d_{k}d_{l}
	\int_{\mathbb{C}}w^{p}\wb^{q}z^{k}
	\wb^{k}\zb^{l}w^{l}e^{-|w|^{m}}dA(w)\\
&=m^{2}\sum_{k,l=0}^{\infty}d_{k}d_{l}z^{k}
	\zb^{l}\int_{\mathbb{C}}w^{p+l}
	\wb^{q+k}e^{-|w|^{m}}dA(w).
\end{align*}
First, assume that $p\geq q$. The above integral is nonzero only 
if $k=p+l-q$. Therefore, 
\begin{equation}\label{EqnProp1}
	\begin{split}
m\sum_{k=0}^{\infty}d_{k}z^{p+k}\zb^{q+k}
&=m^{2}\sum_{l=0}^{\infty}d_{p+l-q}d_{l}z^{p+l-q}
	\zb^{l}\int_{0}^{\infty}r^{2(p+l)+1}e^{-r^{m}}dr\\
&=m^{2}\sum_{l=0}^{\infty}d_{p+l-q}d_{l}z^{p+l-q}
	\zb^{l} \frac{1}{m}\Gamma\left(\frac{2(p+l)+2}{m}\right)\\
&=m\sum_{l=0}^{\infty}d_{p+l-q}d_{l}z^{p+l-q}\zb^{l}\frac{1}{d_{p+l}}.
\end{split}
\end{equation} 
For the above equation to hold, every corresponding $l$-term on each side 
must agree. Note that the zeroth term on the right hand side is holomorphic, 
while there is no holomorphic term on the left hand side. Hence, they 
can not be equal.

When $p<q$, similarly we obtain
\begin{equation*}
m\sum_{k=0}^{\infty}d_{k}z^{p+k}\zb^{q+k} 
=m\sum_{k=0}^{\infty} d_{k}d_{q-p+k}z^{k}\zb^{q-p+k}\frac{1}{d_{q+k}}.  
\end{equation*}
The zeroth term on the right hand side is conjugate holomorphic, 
while there is no conjugate holomorphic term on the left hand side. 
Hence, they are not equal.

Therefore, we can conclude that any monomial which is a fixed point 
of the Berezin transformation should be either of the form $z^{p}$ 
or $\zb^{q}$.
\end{proof}

\begin{proof}[Proof of Theorem \ref{ThmBino}] 
Since the case of monomials was considered in Proposition \ref{PropMono}, 
we can assume that $B_mf=f$ and $a,b,c,d$ are positive integers 
such that $a\neq c$ or $b\neq d$. Then by \eqref{EqnBerezin} we get
\begin{align*}
K_{m}(z,z)(c_{1}z^{a}\zb^{b}+c_{2}z^{c}\zb^{d})=\,&K_{m}(z,z)B_mf(z)\\
=&\int_{\mathbb{C}}(c_{1}w^{a}\wb^{b}+c_{2}w^{c}\wb^{d}) 
|K_{m}(w,z)|^{2}e^{-|w|^{m}}dA(w).
\end{align*}
Therefore, 
\begin{equation}\label{EqnBin1}
\begin{split}
	m\sum_{k=0}^{\infty}
	d_{k}z^{k}\zb^{k}(c_{1}z^{a}\zb^{b}+c_{2}z^{c}\zb^{d}) 
	=\,& m^2\sum_{k,l=0}^{\infty}d_{k}d_{l}
	\int_{\mathbb{C}}(c_{1}w^{a}\wb^{b}+c_{2}w^{c}
	\wb^{d})z^{k}\wb^{k}\zb^{l}w^{l}e^{-|w|^{m}}dA(w)\\
	=\,&c_{1}m^2\sum_{k,l=0}^{\infty}d_{k}d_{l}z^{k}
	\zb^{l}\int_{\mathbb{C}}w^{a+l}\wb^{b+k}e^{-|w|^{m}}dA(w)\\
	&+c_{2}m^2\sum_{k,l=0}^{\infty}d_{k}d_{l}z^{k}\zb^{l}
	\int_{\mathbb{C}}w^{c+l}\wb^{d+k}e^{-|w|^{m}}dA(w).
\end{split}
\end{equation}

We consider two different cases $a\geq b\geq 1,c\geq d\geq 1$ and 
$a\geq b\geq 1,1\leq c<d$. The other cases turn into one of these 
simply by conjugation.

First let us consider the case $a\geq b\geq 1$ and $c\geq d\geq 1$.  
Note that the first integral on the right hand side of \eqref{EqnBin1} 
is nonzero only if $k=a+l-b$, and the second integral in nonzero only 
if $k=c+l-d$. Then similarly to the monomial case \eqref{EqnProp1}, 
we can write the right hand side of \eqref{EqnBin1} as
\begin{equation*}
c_1m\sum_{l=0}^{\infty}d_{a-b+l}d_{l}z^{a-b+l}\zb^l \frac{1}{d_{a+l}}
+c_2m\sum_{l=0}^{\infty}d_{c-d+l}d_lz^{c-d+l}\zb^l \frac{1}{d_{c+l}}.
\end{equation*}
Hence if $f$ is a fixed point, then
\begin{equation}\label{EqnEqu1}
\begin{split}
	c_1\sum_{l=0}^{\infty}d_lz^{a+l}\zb^{b+l}
	+c_2\sum_{l=0}^{\infty} d_lz^{c+l}\zb^{d+l}
	=c_1\sum_{l=0}^{\infty}\frac{d_{a-b+l}d_l}{d_{a+l}}z^{a-b+l}\zb^l
	+c_2\sum_{l=0}^{\infty}\frac{d_{c-d+l}d_l}{d_{c+l}}z^{c-d+l}\zb^l. 
\end{split} 
\end{equation}
Let $a-b=j$ and $c-d=k$. We can write (\ref{EqnEqu1}) as
\begin{equation*} 
c_1\sum_{l=0}^{\infty}d_lz^{a+l}\zb^{b+l}
+c_2\sum_{l=0}^{\infty}d_lz^{c+l}\zb^{d+l} 
=c_1\sum_{l=0}^{\infty}\frac{d_{j+l}d_l}{d_{a+l}}z^{j+l}\zb^l
+c_2\sum_{l=0}^{\infty}\frac{d_{k+l}d_l}{d_{c+l}}z^{k+l}\zb^l.
\end{equation*}
Note that the zeroth terms in the sums on the right hand side 
above are holomorphic, while there is no holomorphic 
term on the left hand side. Hence if $j\neq k$ and $f$ is a fixed point, 
the zeroth terms on the right hand side must be zero, implying that
\begin{equation*}
c_{1}\frac{d_{j}d_{0}}{d_{a}}=0\quad  \text{ and } \quad 
c_{2}\frac{d_{k}d_{0}}{d_{c}}=0.
\end{equation*}
Since $d_{i}\neq 0$ for any $i\geq 0$, we can conclude 
that $c_{1}=c_{2}=0$. 

From now on we assume 
that $j=k$. Then, we take $B_{m}f=f$ with $a-b=c-d=j\geq 0$. Hence  
(\ref{EqnEqu1}) can be written as
\begin{equation}\label{EqnEqu2}
c_1\sum_{l=0}^{\infty}d_{l}z^{a+l}\zb^{b+l}
+c_2\sum_{l=0}^{\infty}d_{l}z^{c+l}\zb^{d+l}
=\sum_{l=0}^{\infty}d_{j+l}d_l
\left(\frac{c_{1}}{d_{a+l}}+\frac{c_{2}}{d_{c+l}}\right)z^{j+l}\zb^l. 
\end{equation}
Again the zeroth term on the right hand side above is holomorphic, 
while there is no holomorphic term on the left hand side. 
So for $B_{m}f=f$ to hold, we must assume that the holomorphic 
term on the right hand side is zero. This is equivalent to
\begin{equation}\label{EqnEqu3}
\frac{c_{1}}{d_{a}}+\frac{c_{2}}{d_{c}}=0.
\end{equation}
Writing $a=j+b$ and $c=j+d$, we can write (\ref{EqnEqu2}) as
\begin{equation}\label{EqnEqu4} 
c_1\sum_{l=0}^{\infty}d_lz^{j+b+l}\zb^{b+l}
+c_2\sum_{l=0}^{\infty}d_lz^{j+d+l}\zb^{d+l}
=\sum_{l=0}^{\infty}d_{j+l}d_l
\left(\frac{c_{1}}{d_{a+l}}+\frac{c_{2}}{d_{c+l}}\right)z^{j+l}\zb^l.
\end{equation}
Without loss of generality, we  assume that $d>b$ and we define 
$r=d-b\geq 1$. Recall that when $r=0$, $f$ is a monomial as $a-b=c-d$ and  
that was considered in Proposition \ref{PropMono}. Hence (\ref{EqnEqu4}) 
can be written as the following.
\begin{equation}\label{EqnEqu5} 
c_1\sum_{l=0}^{\infty}d_{l}z^{j+b+l}\zb^{b+l}
+c_2\sum_{l=0}^{\infty}d_{l}z^{j+b+r+l}\zb^{b+r+l}
=\sum_{l=0}^{\infty}d_{j+l}d_{l}
\left(\frac{c_{1}}{d_{a+l}}+\frac{c_{2}}{d_{c+l}}\right)z^{j+l}\zb^l.
\end{equation}
To get a better idea about the series above, we can write (\ref{EqnEqu5}) as
\begin{align*}
c_{1}d_{0}z^{j+b}\zb^{b}&+c_{1}d_{1}z^{j+b+1}\zb^{b+1}
+\cdot\cdot\cdot+c_{1}d_{r-1}z^{j+b+r-1}\zb^{b+r-1}\\
&+c_{1}\sum_{l=r}^{\infty}d_{l}z^{j+b+l}\zb^{b+l}+c_{2}
\sum_{l=0}^{\infty}d_{l}z^{j+b+r+l}\zb^{b+r+l}\\	
=d_{j}d_{0}&\left(\frac{c_{1}}{d_{a}}
+\frac{c_{2}}{d_{c}}\right)z^{j}+d_{j+1}d_{1}
\left(\frac{c_{1}}{d_{a+1}}+\frac{c_{2}}{d_{c+1}}\right)
z^{j+1}\zb+\cdot\cdot\cdot\\
&+d_{j+b+r-1}d_{b+r-1}\left(\frac{c_{1}}{d_{a+b+r-1}}
+\frac{c_{2}}{d_{c+b+r-1}}\right)z^{j+b+r-1}\zb^{b+r-1}\\
&+\sum_{l=b+r}^{\infty}d_{j+l}d_{l}
\left(\frac{c_{1}}{d_{a+l}}+\frac{c_{2}}{d_{c+l}}\right)
z^{j+l}\zb^{l}.
\end{align*}
Then we can rewrite it as
\begin{align*}
c_{1}d_{0}z^{j+b}\zb^{b}&+c_{1}d_{1}z^{j+b+1}\zb^{b+1}
+\cdot\cdot\cdot+c_{1}d_{r-1}z^{j+b+r-1}\zb^{b+r-1}\\
&+c_{1}\sum_{l=0}^{\infty}d_{r+l}z^{j+b+r+l}\zb^{b+r+l} 
+c_{2}\sum_{l=0}^{\infty}d_{l}z^{j+b+r+l}\zb^{b+r+l}\\
=d_{j}d_{0}&\left(\frac{c_{1}}{d_{a}}+\frac{c_{2}}{d_{c}}\right)z^{j}
+d_{j+1}d_{1}\left(\frac{c_{1}}{d_{a+1}}+\frac{c_{2}}{d_{c+1}}\right)
z^{j+1}\zb+\cdot\cdot\cdot\\
&+d_{j+b+r-1}d_{b+r-1}\left(\frac{c_{1}}{d_{a+b+r-1}}
+\frac{c_{2}}{d_{c+b+r-1}}\right)z^{j+b+r-1}\zb^{b+r-1}\\
&+\sum_{l=0}^{\infty}d_{j+b+r+l}d_{b+r+l}
\left(\frac{c_{1}}{d_{a+b+r+l}}+\frac{c_{2}}{d_{c+b+r+l}}\right)
z^{j+b+r+l}\zb^{b+r+l},
\end{align*}
where the sums on both the left and right hand sides are in terms 
of $z^{j+b+r+l}\zb^{b+r+l}$. Comparing the coefficients 
of $z^{j+b+r-1}\zb^{b+r-1}$ on both sides, we obtain
\begin{equation}\label{EqnEqu6}
c_{1}d_{r-1}=d_{j+b+r-1}d_{b+r-1}
\left(\frac{c_{1}}{d_{a+b+r-1}}+\frac{c_{2}}{d_{c+b+r-1}}\right)
\end{equation}
Comparing later terms, we must have
\begin{align*}
c_{1}d_{r+l}+c_{2}d_{l}
=d_{j+b+r+l}d_{b+r+l}
\left(\frac{c_{1}}{d_{a+b+r+l}}	+\frac{c_{2}}{d_{c+b+r+l}}\right)
\text{ for all }l\in \mathbb{N}_{0}.
\end{align*}

Considering (\ref{EqnEqu3}) and (\ref{EqnEqu6}), we will show that 
$c_{1}=c_{2}=0$. Equivalently, we would like to show that the 
following determinant 
\begin{equation*}
A_{m}=\det   
\left[ {\begin{array}{cc}
		\frac{1}{d_{a}} 
		& \frac{1}{d_{c}} \\
		\frac{d_{j+b+r-1}d_{b+r-1}}{d_{a+b+r-1}}-d_{r-1} 
		& \frac{d_{j+b+r-1}d_{b+r-1}}{d_{c+b+r-1}} \\
\end{array} } \right] 
\end{equation*}
is non-zero. Using $d_{l}=\frac{1}{\Gamma\left(\frac{2l+2}{m}\right)}$, 
we can write the determinant above as
\begin{align*}
A_{m}&=\frac{ d_{j+b+r-1}d_{b+r-1}}{d_{a}d_{c+b+r-1}}
-\frac{ d_{j+b+r-1}d_{b+r-1}}{d_{c}d_{a+b+r-1}}
+\frac{d_{r-1}}{d_{c}}\\
&=\frac{\Gamma\left(\frac{2a+2}{m}\right)\Gamma
	\left(\frac{2(c+b+r-1)+2}{m}\right)}{\Gamma\left(
	\frac{2(j+b+r-1)+2}{m}\right)\Gamma
	\left(\frac{2(b+r-1)+2}{m}\right)}
-\frac{\Gamma\left(\frac{2c+2}{m}\right)\Gamma
	\left(\frac{2(a+b+r-1)+2}{m}\right)}{\Gamma
	\left(\frac{2(j+b+r-1)+2}{m}\right)\Gamma
	\left(\frac{2(b+r-1)+2}{m}\right)}+\frac{\Gamma
	\left(\frac{2c+2}{m}\right)}{\Gamma\left(\frac{2(r-1)+2}{m}\right)}.
\end{align*}
Letting $x=\frac{2}{m}$, we get
\begin{equation*}
A_{m}=\frac{\Gamma((a+1)x)\Gamma((c+b+r)x)}{
	\Gamma((j+b+r)x)\Gamma((b+r)x)}
-\frac{\Gamma((c+1)x)\Gamma((a+b+r)x)}{
	\Gamma((j+b+r)x)\Gamma((b+r)x)}+\frac{\Gamma((c+1)x)}{\Gamma(rx)}.
\end{equation*}
One can show that the contribution of the first two terms is already
positive. Indeed,
\begin{equation*}
\frac{\Gamma((a+1)x)\Gamma((c+b+r)x)}{
	\Gamma((j+b+r)x)\Gamma((b+r)x)}
-\frac{\Gamma((c+1)x)\Gamma((a+b+r)x)}{\Gamma((j+b+r)x)\Gamma((b+r)x)}>0
\end{equation*}
if and only if 
\[\Gamma((a+1)x)\Gamma((c+b+r)x) -\Gamma((c+1)x)\Gamma((a+b+r)x)>0\]
which is equivalent to 
\begin{equation*}
\frac{\Gamma((a+1)x)\Gamma((c+b+r)x)}{\Gamma((a+b+c+r+1)x)}
>\frac{\Gamma((c+1)x)\Gamma((a+b+r)x)}{\Gamma((a+b+c+r+1)x)}. 
\end{equation*}  
However, using $(a+1)x+(c+b+r)x=(c+1)x+(a+b+r)x=k$ we have 
\begin{align*}
\beta((a+1)x,k-(a+1)x)
=&\frac{\Gamma((a+1)x)\Gamma((c+b+r)x)}{\Gamma((a+b+c+r+1)x)}\\
\beta((c+1)x,k-(c+1)x)
&=\frac{\Gamma((c+1)x)\Gamma((a+b+r)x)}{\Gamma((a+b+c+r+1)x)}. 
\end{align*} 
We recall that $a<c$. So we would like to show that
\begin{equation}\label{EqnEqu7}
\beta((a+1)x,k-(a+1)x)>\beta((c+1)x,k-(c+1)x).
\end{equation}
By Lemma \ref{LemBeta}, since $\beta(y,k-y)$ is convex with minimum at  
$k/2=(a+b+c+r+1)x/2$,  it is enough to show that 
\[(a+1)x< (c+1)x\leq \frac{k}{2}.\]
The first inequality above is clear as $a<c$. 
The second inequality is equivalent to 
\[c+1 \leq a+b+r=a+d\] 
as $r=d-b$. However, $a-b=c-d$. Then the inequality above is equivalent to 
\[a-b+1\leq a\]
which is correct as $b\geq 1$. Hence, the inequality \eqref{EqnEqu7} is satisfied
and $A_m$ is non-singular for all $m>0$. Therefore, $c_1=c_2=0$.  

We finish the proof by considering the second case $a\geq b\geq 1$ and 
$1\leq c<d$.  The first integral on the right hand side of \eqref{EqnBin1} 
is nonzero when $k=a-b+l$ and the second integral 
is nonzero when $l=d-c+k$. Therefore, \eqref{EqnBin1} can be written as
\begin{equation}\label{EqnBin2}
\begin{split}
	&c_{1} \sum_{k=0}^{\infty} d_{k}z^{a+k}\zb^{b+k}
	+c_{2} \sum_{k=0}^{\infty} d_{k}z^{c+k}\zb^{d+k}\\
	&=c_{1}\sum_{l=0}^{\infty}\frac{d_{a-b+l}d_{l}}{d_{a+l}}
	z^{a-b+l}\zb^{l}+c_{2}\sum_{k=0}^{\infty}
	\frac{d_{k}d_{d-c+k}}{d_{d+k}}z^{k}\zb^{d-c+k}.
\end{split}
\end{equation}
The zeroth terms in the sums on the right hand side are harmonic, 
while there is no harmonic term on the left hand side. 
Hence \eqref{EqnBin2} cannot hold unless $c_{1}=c_{2}=0$.
\end{proof}

\end{document}